\documentclass[12pt,a4paper]{article}

\usepackage[margin=2.5cm]{geometry}
\usepackage{amsmath,amsfonts,amsthm}
\usepackage{graphicx,color,parskip}
\usepackage{enumitem}
\usepackage{mathtools}
\usepackage[english]{babel}
\usepackage{todonotes}

\newcommand{\N}{\mathbb{N}}

\newcommand{\R}{\mathbb{R}}

\DeclarePairedDelimiter\norm{\lVert}{\rVert}

\newcommand{\rinj}[1]{r_\text{inj}(#1)}

\newcommand{\T}{{\rm T}}
\newcommand{\D}{{\rm D}}

\newtheorem{theorem}{Theorem}
\newtheorem{lemma}[theorem]{Lemma}

\newtheorem{definition}[theorem]{Definition}
\newtheorem{remark}[theorem]{Remark}

\title{Notions of uniform manifolds}
\author{Jaap Eldering}

\begin{document}

\maketitle

Let $M$ be an $n$-dimensional manifold; the assumed smoothness will be
clear from the context. We are interested in uniformity properties of
$M$ when it is noncompact. These can be formulated in different
ways, e.g.\ in terms of bounded geometry when a Riemannian metric $g$
is present. If no such metric is (canonically) available, it may be more
natural to express uniformity in terms of the atlas and its chart
transition maps. We shall formulate various definitions of uniformity
and investigate their relations.

We follow~\cite{Eichhorn1991:mfld-metrics-noncpt} to define bounded
geometry.
\begin{definition}[Bounded geometry]
  \label{def:bound-geom}
  We say that a complete, finite-dimensional Riemannian manifold
  $(M,g)$ has $k$-th order bounded geometry when the following
  conditions are satisfied:
  \begin{description}[labelindent=1ex,labelwidth=4ex,leftmargin=!]
  \item[(I)] the global injectivity radius $\rinj{M} = \inf\limits_{x \in M}\; \rinj{x}$
    is positive, $\rinj{M} > 0$;
  \item[(B$_k$)] the Riemannian curvature $R$ and its covariant
    derivatives up to $k$-th order are uniformly bounded,
    \begin{equation*}
      \forall\; 0\le i \le k\colon \sup_{x \in M}\; \norm{\nabla^i R(x)} < \infty,
    \end{equation*}
    with operator norm of\/ $\nabla^i R(x)$ as an element of the
    tensor bundle over $x \in M$.
  \end{description}
\end{definition}

\begin{remark}
  Condition \textbf{I} already automatically implies that $(M,g)$ is a
  complete metric space.
\end{remark}

This formulation is also called \emph{coordinate-free} bounded
geometry. It is shown in~\cite{Eichhorn1991:bound-conncoef} (and
in~\cite{Schick2001:mlfd-bndry-boundgeom} for manifolds with boundary)
that this definition implies \emph{coordinate-wise defined} bounded
geometry, where condition \textbf{B$_k$} is replaced by \textbf{B$_k'$}: there exists a
radius $0 < r_0 < \rinj{M}$ such that on each normal coordinate chart
of radius $r_0$, the metric coefficients $g_{ij}$ and their
derivatives up to order $k$ are bounded by a global constant $C_k$.
In~\cite[Prop.~2.4]{Roe1988:index-openmflds} a proof is sketched that
the converse also holds when $k = \infty$. Note that for finite $k$ we
would incur a loss of at least two degrees of differentiability, since
the curvature is defined in terms of second derivatives of the metric.

Let us introduce a more general notion of uniformity, defined purely
in terms of the atlas of a smooth manifold.
\begin{definition}[Uniform manifold]
  \label{def:unif-manifold}
  We say that a manifold $M$ with atlas
  $\mathcal{A} = \big\{(\phi_i\colon U_i \to \R^n) \mid i \in I\big\}$
  is uniform of order $k \ge 1$ if
  \begin{description}[labelindent=1ex,labelwidth=4ex,leftmargin=!]
  \item[(I)] there exists one uniform $\delta > 0$ such that for each
    $x \in M$ there exists a coordinate chart $\phi_i$ that covers a
    ball of radius $\delta$ around $x$, i.e.
    \begin{equation}
      B(\phi_i(x);\delta) \subset \phi_i(U_i);
    \end{equation}

  \item[(B$_k$)] there is one global bound $B_k$ such that all
    transition maps are uniformly bounded in $C^k$ norm:
    \begin{equation}
      \forall i,j \in I\colon \norm{\phi_j\circ\phi_i^{-1}}_k \le B_k.
    \end{equation}
  \end{description}
\end{definition}

\begin{remark}
  The $C^0$ part of the $C^k$ bound restricts charts to have uniformly
  bounded image $\phi(U) \subset \R^n$. This is no loss of generality,
  since we can always break up a large chart and translate each part
  close to the origin. Alternatively, we could require the bound $B_k$
  for the derivatives of $(\phi_j\circ\phi_i^{-1})$ only.
\end{remark}

\begin{definition}[Uniformly compatible atlases]
  \label{def:unif-compat-atlas}
  Let $\mathcal{A}, \mathcal{A'}$ be two uniform $C^k$ atlases for the
  manifold $M$. We say that these are uniformly compatible if the
  union $\mathcal{A} \cup \mathcal{A}'$ is again a uniform $C^k$ atlas
  for $M$.
\end{definition}

Note that although this definition looks identical to the standard
definition for (non-uniform) compatibility of atlases, it does
implicitly depend on the parameters $\delta$ and $B_k$ in
Definition~\ref{def:unif-manifold}. The `injectivity radius' $\delta$
of the combined atlas will at least be equal to the maximum of the
radii of $\mathcal{A}$ and $\mathcal{A}'$, but the bound $B_k$ of the
combination may be larger than the maximum of their bounds. A maximal
uniform atlas cannot be defined since it would require fixing a $B_k$,
but this would mean that uniform compatibility of charts is not an
equivalence relation anymore.

Classes of uniformly bounded $C^k$ functions are defined as follows.
\begin{definition}[Uniformly bounded $C^k$ functions]
  \label{def:unif-Ck-functions}
  Let $(M,\mathcal{A})$ and $(N,\mathcal{B})$ be uniform manifolds of
  order $k \ge 1$. Then we define the class $C^k_b(M;N)$ of uniformly
  bounded $C^k$ functions to consist of those functions
  $f \in C^k(M;N)$ for which there exists a bound $C > 0$ such that
  the coordinate representations satisfy
  \begin{equation}\label{eq:unif-Ck-function}
    \norm{\psi \circ f \circ \phi^{-1}}_k \le C
  \end{equation}
  for all charts $\phi \in \mathcal{A}$ and $\psi \in \mathcal{B}$
  (where defined).
\end{definition}
Note that this class of functions is closed under composition, and
under and multiplication e.g.\ when $N = \R$.

It turns out that definitions~\ref{def:bound-geom}
and~\ref{def:unif-manifold} are equivalent in the following sense.
\begin{theorem}
  \label{thm:equiv-uniform}
  Let $(M,\mathcal{A})$ be a uniform manifold. Then there exists a
  metric $g$ such that $(M,g)$ has bounded geometry and $g$ induces an
  atlas $\mathcal{A}'$ of normal coordinate charts which is uniformly
  compatible again with $\mathcal{A}$.
\end{theorem}

The order of smoothness may decrease by a small amount in the process
$\mathcal{A} \rightsquigarrow g \rightsquigarrow \mathcal{A}'$. We
shall work through a number of lemmas to prove this result.

\begin{lemma}\label{lem:BG-implies-unif}
  Let $(M,g)$ be a Riemannian manifold of $k$-th order
  (coordinate-free defined) bounded geometry with $k \ge 2$. Then $M$
  is a uniform manifold of order $k-1$ with the preferred atlas given
  by the normal coordinate charts of some radius $\delta > 0$.
\end{lemma}

See~\cite[Lem.~2.6]{Eldering2013:NHIM-noncompact} for a proof. This
result shows that the atlas of normal coordinate charts that arises
from a manifold $(M,g)$ of bounded geometry is itself uniform.

\begin{remark}
  We need not necessarily add the normal coordinate charts centered
  around all points $x \in M$ to the preferred atlas. If we instead
  constructed a (uniformly locally finite) cover as
  in~\cite[Lem.~2.6]{Eldering2013:NHIM-noncompact} with normal coordinate
  balls of size $\delta_2$, such that the balls of a fixed size
  $\delta_1 < \delta_2$ already cover $M$, then this satisfies
  Definition~\ref{def:unif-manifold} with
  $\delta = \delta_2 - \delta_1$.
\end{remark}

We introduce some notation and intermediate results based
on~\cite[Sect.~2.1]{Eldering2013:NHIM-noncompact}. For any point
$x \in M$, $\phi_x$ will denote a chosen coordinate chart that
satisfies condition~\textbf{I} of Definition~\ref{def:unif-manifold}.
Then we define the open neighborhood
\begin{equation}\label{eq:ball}
  B_x(\delta') = \phi_x^{-1}\big(B(\phi_x(x);\delta')\big) \subset M
\end{equation}
as the preimage of the ball of radius $\delta' \le \delta$ around
$\phi_x(x)$. Note that $B_x(\delta')$ need not be a ball in another
coordinate chart $\phi$, but since $\D(\phi\circ\phi_x^{-1})$ and its
inverse are bounded by $B_k$, we have that
\begin{equation}\label{eq:ball-uniformity}
  B\big(\phi(x);\delta'/B_k\big) \subset \phi\big(B_x(\delta)\big) \subset
  B\big(\phi(x);B_k\,\delta'\big)
\end{equation}
insofar these lie within the image of $\phi$, that is, coordinate
chart transformations deform balls only boundedly so. In other words:
local distances induced by the Euclidean distance in each of the
charts are equivalent up to a factor $B_k$. Then we have the
following result.
\begin{lemma}[Uniformly locally finite cover]
  \label{lem:unif-loc-cover}
  Let $M$ be a uniform manifold of order $k \ge 1$ with parameters
  $\delta$ and $B_k$.

  Then for $0 < \delta_1 < \delta_2 < \delta$ small enough, $M$ has a
  countable cover $\big\{B_{x_i}(\delta_2)\big\}_{i \ge 1}$ such that
  \begin{enumerate}
  \item the sets $B_{x_i}(\delta_1)$ already cover $M$;
  \item $\forall\; j \neq i\colon x_j \not\in B_{x_i}(\delta_1/B_k)$;
  \item there exists an explicit global bound $K \in \N$ such that for
    each $x \in M$ the neighborhood $B_x(\delta_2)$ intersects at most
    $K$ of the $B_{x_i}(\delta_2)$.
  \end{enumerate}
\end{lemma}

We follow the proof of~\cite[Lem.~2.16]{Eldering2013:NHIM-noncompact}
with appropriate modifications to replace the metric setting there.

\begin{proof}
  Assume $\delta_1,\delta_2$ fixed, these will be determined later.
  Let $\{M_k\}_{k \ge 1}$ be an exhaustion of $M$ by compact sets.
  Cover $M_k$ by a finite sequence of balls $B_{x_i}(\delta_1)$ that
  extend the sequence covering $M_{k-1}$, as follows: choose a point
  $x \in M_k$ that is not covered yet, and add $B_x(\delta_1)$ to the
  sequence. This sequence is finite, for if it were infinite, it would
  have a converging subsequence $x_{i_j} \to \bar{x} \in M_k$. This is
  a contradiction since then
  $\norm{\phi_{\bar{x}}(x_{i_j}) - \phi_{\bar{x}}(\bar{x})} \to 0$ and
  uniform equivalence of the Euclidean distances in charts now implies
  that the distance between points $x_{i_{j'}}$ in charts
  $\phi_{x_{i_j}}$ must converge to zero as $j,j' \to \infty$. This
  contradicts the assumption that new points $x_i$ are not being
  covered yet. Since $x_j \not\in B_{x_i}(\delta_1)$ for $i<j$ implies
  that $x_i \not\in B_{x_j}(\delta_1/B_k)$, the limit of these
  sequences satisfies the first two claims of the lemma.

  For the third claim let $x \in M$ be arbitrary.
  By~\eqref{eq:ball-uniformity} any ball $B_{x_i}(\delta_2)$ that
  intersects $B_x(\delta_2)$ must be completely contained in
  $B_x\big((1+2B_k)\delta_2\big)$, where we set
  $\delta_2 < \delta/(1+2B_k)$ to have this well-defined. In the chart
  $\phi_x$, each $B_{x_i}(\delta_2)$ occupies an exclusive set
  $B(\phi_x(x_i);\delta_2/B_k)$ with respect to any of the other
  $B_{x_{i'}}(\delta_2)$ balls. Thus, by considering volume estimates
  in the chart $\phi_x$ we obtain
  \begin{equation}\label{eq:bound-K}
    K \le \frac{\big((1+2B_k)\delta_2\big)^n}{(\delta_2/B_k)^n}
      \le \big(\sqrt{3}\,B_k\big)^{2n}
  \end{equation}
  using that $B_k \ge 1$.
\end{proof}

The following lemma is a
straightforward adaptation of~\cite[Lem.~2.17]{Eldering2013:NHIM-noncompact}.
\begin{lemma}[Uniform partition of unity]
  \label{lem:part-unity}
  Let $M$ be a uniform manifold with a uniformly locally finite cover
  with $0 < \delta_1 < \delta$ as per Lemma~\ref{lem:unif-loc-cover}.

  Then there exists a partition of unity by functions
  $\chi_i \in C^k_b(B_{x_i}(\delta_2);[0,1])$ subordinate to this cover.
  There is a global bound on the $C^k$ norm of all the functions
  $\chi_i$.
\end{lemma}

\begin{proof}[Proof of Theorem~\ref{thm:equiv-uniform}]
  We first adapt a standard method to construct a Riemannian
  metric $g$, and then prove both that this metric has bounded
  geometry and that its normal coordinate chart atlas $\mathcal{A}'$
  is uniformly compatible with the original uniform atlas
  $\mathcal{A}$.

  Let $\{(\phi_i,B_{x_i}(\delta_2))\}_{i \in \N}$ be a uniformly locally finite
  cover of $M$ with subordinate partition of unity by functions
  $\chi_i$. Let $g_e$ denote the Euclidean metric on $\R^n$. We define
  the metric
  \begin{equation}\label{eq:unif-metric}
    g = \sum_{i \in \N} \chi_i\,\phi_i^*(g_e).
  \end{equation}
  Note that $g \in C^{k-1}_b$ since at most $K$ terms in the sum are
  non-zero and each term is a composition and product of $C^{k-1}_b$
  functions $\chi_i, \phi_i, \D\phi_i$ and $g_e$; it is a sum of
  positive-definite bilinear forms, hence again positive-definite and
  invertible everywhere. Let $\phi \in \mathcal{A}$ be a coordinate
  chart and $v \in \R^n$ with $\norm{v} = 1$, then we have a lower
  bound
  \begin{equation*}
    \begin{aligned}
      (\phi_* g)(v,v)
      &= \sum_{i \in \N} (\chi_i\circ\phi^{-1})\cdot(\phi_i\circ\phi^{-1})^*(g_e)(v,v)\\
      &= \sum_{i \in \N} (\chi_i\circ\phi^{-1})\,\norm{\D(\phi_i\circ\phi^{-1})\,v}^2\\
      &\ge \frac{1}{B_k^2},
    \end{aligned}
  \end{equation*}
  hence $g^{-1}$ is bounded by $B_k^2$ in any chart. From the
  expression of the derivatives of $g^{-1}$ in terms of $g^{-1}$
  itself and derivatives of $g$, it follows that
  $g^{-1} \in C^{k-1}_b$; thus, also the Christoffel symbols satisfy
  $\Gamma \in C^{k-2}_b$ and this also proves that condition
  \textbf{B}$_{k-3}$ of coordinate-free bounded geometry is satisfied.

  To prove condition~\textbf{I} of Definition~\ref{def:bound-geom}
  that $g$ has a finite injectivity radius, and finally that the
  original atlas $\mathcal{A}$ and the atlas $\mathcal{A}'$ of normal
  coordinate charts are uniformly compatible, we consider coordinate
  transition maps from charts $\exp_x^{-1} \in \mathcal{A}'$ to
  $\phi \in \mathcal{A}$. Since $\exp$ is defined through the time-one
  geodesic flow, we can view a transition map
  \begin{equation}
    \phi^{-1}\circ\exp_x\colon \T_x M \to \R^n
  \end{equation}
  as a local coordinate expression of
  $\exp_x = \pi \circ \Upsilon^1|_{\T_x M}$, where $\Upsilon^t$
  denotes the geodesic flow on $\T M$. This flow is defined by the
  differential equation
  \begin{equation}\label{eq:ODE-geod-flow}
    \begin{aligned}
      \dot{x}^i &= v^i,\\
      \dot{v}^i &= \Gamma(x)^i_{jk} v^j v^k,
    \end{aligned}
  \end{equation}
  with respect to coordinates $(x^i,v^j)$ on $\T M$ induced by $\phi$.
  The flow $\Upsilon^t$ preserves $\norm{v}^2 = g_x(v,v)$ and
  $g,g^{-1}$ are uniformly bounded, so also the coordinate expressions
  $v^i$ are uniformly bounded by $B_k^2\,\delta'$ when the initial
  value $v(0)$ is bounded by $\delta'$. Since $\Gamma$ is bounded
  also, we have that~\eqref{eq:ODE-geod-flow} approximates the system
  \begin{equation}\label{eq:ODE-eucl-flow}
    \begin{aligned}
      \dot{x}^i &= v^i,\\
      \dot{v}^i &= 0,
    \end{aligned}
  \end{equation}
  when $\norm{v(0)} \le \delta'$ is small. Note
  that~\eqref{eq:ODE-eucl-flow} induces the `Euclidean exponential map'
  $\widetilde{\exp}_x(v) = x + v$ with respect to the chart $\phi$.
  Thus by uniform dependence of a flow on the vector field
  (see~\cite[Thm.~A.6]{Eldering2013:NHIM-noncompact}) it follows that
  $\phi^{-1}\circ\exp_x =
  \phi^{-1}\circ\pi\circ\Upsilon^1(x,\,\cdot\,)$
  approximates the identity map on $\R^n$ in $C^{k-2}$ norm for
  sufficiently small $\delta'$. Hence the transition maps are uniform
  $C^{k-2}$ diffeomorphisms and thus the altases $\mathcal{A}$ and
  $\mathcal{A}'$ are uniformly compatible. This also shows that the
  injectivity radius of $g$ satisfies $\rinj{M} \ge \delta'$.
\end{proof}

\bibliographystyle{amsalpha}
\bibliography{bibfile}

\def\polhk#1{\setbox0=\hbox{#1}{\ooalign{\hidewidth
  \lower1.5ex\hbox{`}\hidewidth\crcr\unhbox0}}} \def\cprime{$'$} \def\ui{i}
\providecommand{\bysame}{\leavevmode\hbox to3em{\hrulefill}\thinspace}
\providecommand{\MR}{\relax\ifhmode\unskip\space\fi MR }
\providecommand{\MRhref}[2]{%
  \href{http://www.ams.org/mathscinet-getitem?mr=#1}{#2}
}
\providecommand{\href}[2]{#2}
\begin{thebibliography}{Eic91b}

\bibitem[Eic91a]{Eichhorn1991:mfld-metrics-noncpt}
J{\"u}rgen Eichhorn, \emph{The {B}anach manifold structure of the space of
  metrics on noncompact manifolds}, Differential Geom. Appl. \textbf{1} (1991),
  no.~2, 89--108. \MR{MR1244437 (94j:58028)}

\bibitem[Eic91b]{Eichhorn1991:bound-conncoef}
\bysame, \emph{The boundedness of connection coefficients and their
  derivatives}, Math. Nachr. \textbf{152} (1991), 145--158. \MR{1121230
  (92k:53069)}

\bibitem[Eld13]{Eldering2013:NHIM-noncompact}
Jaap Eldering, \emph{Normally hyperbolic invariant manifolds --- the noncompact
  case}, Atlantis Series in Dynamical Systems, vol.~2, Atlantis Press, Paris,
  September 2013. \MR{3098498}

\bibitem[Roe88]{Roe1988:index-openmflds}
John Roe, \emph{An index theorem on open manifolds. {I}, {II}}, J. Differential
  Geom. \textbf{27} (1988), no.~1, 87--113, 115--136. \MR{918459 (89a:58102)}

\bibitem[Sch01]{Schick2001:mlfd-bndry-boundgeom}
Thomas Schick, \emph{Manifolds with boundary and of bounded geometry}, Math.
  Nachr. \textbf{223} (2001), 103--120. \MR{1817852 (2002g:53056)}

\end{thebibliography}

\end{document}